\documentclass[preprint,3p,10pt,authoryear]{elsarticle}

\journal{Statistics \& Probability Letters}

\usepackage{amssymb}
\usepackage{amsthm}
\usepackage{lineno}  

\usepackage{enumerate,graphicx,mathtools,natbib,geometry,xcolor}
\usepackage[colorlinks]{hyperref}

\newtheorem{theorem}{Theorem}[section]
\newtheorem{proposition}[theorem]{Proposition}
\newtheorem{lemma}[theorem]{Lemma}
\newtheorem{remark}[theorem]{Remark}

\makeatletter \@addtoreset{equation}{section} \makeatother

\newcommand{\N}{\mathbb{N}}
\newcommand{\R}{\mathbb{R}}
\newcommand{\PP}{\mathbb{P}}
\newcommand{\EE}{\mathbb{E}}
\newcommand{\nvert}[0]{\, \vert \, }
\newcommand{\bb}[1]{\boldsymbol{#1}}

\allowdisplaybreaks

\begin{document}

\begin{frontmatter}

    \title{A uniform $L^1$ law of large numbers for functions of i.i.d.\hspace{-1mm} random variables that are translated by a consistent estimator}

    \author[a1]{Pierre Lafaye de Micheaux}
    \author[a2]{Fr\'ed\'eric Ouimet\corref{cor1}\fnref{fn1}}

    \address[a1]{School of Mathematics and Statistics, UNSW Sydney, Australia.}
    \address[a2]{D\'epartement de Math\'ematiques et de Statistique, Universit\'e de Montr\'eal, Canada.}

    \cortext[cor1]{Corresponding author}
    \ead{ouimetfr@dms.umontreal.ca}

    \fntext[fn2]{F. Ouimet is supported by a NSERC Doctoral Program Alexander Graham Bell scholarship (CGS D3).}

    \begin{abstract}
        We develop a new $L^1$ law of large numbers where the $i$-th summand is given by a function $h(\cdot)$ evaluated at $X_i - \theta_n$,
        and where $\theta_n \circeq \theta_n(X_1,X_2,\ldots,X_n)$ is an estimator converging in probability to some parameter $\theta\in \R$.
        Under broad technical conditions, the convergence is shown to hold uniformly in the set of estimators interpolating between $\theta$ and another consistent estimator $\theta_n^{\star}$.
        Our main contribution is the treatment of the case where $|h|$ blows up at $0$, which is not covered by standard uniform laws of large numbers.
    \end{abstract}

    \begin{keyword}
        uniform law of large numbers \sep Taylor expansion \sep M-estimators \sep score function
        \MSC[2010]{60F25}
    \end{keyword}

\end{frontmatter}

\section{Introduction}\label{sec:intro}

    Let $X_1, X_2, X_3,\ldots$ be a sequence of i.i.d.\ random variables and consider the statistic $T_n(\theta_n^{\star})$ where the random variable
    \begin{equation*}
        T_n(\theta) \circeq T_n(X_1,X_2,\ldots,X_n; \theta) : \Omega \to \R
    \end{equation*}
    depends on an unknown parameter $\theta\in \R$ for which we have a consistent sequence of estimators $\theta_n^{\star} \circeq \theta_n^{\star}(X_1,X_2,\ldots,X_n)$.
    Assume further that the following first-order Taylor expansion is valid :
    \begin{equation}\label{eq:Taylor.expansion}
        T_n(\theta_n^{\star}) = T_n(\theta) + (\theta_n^{\star} - \theta) \int_0^1 T_n'(\theta + v(\theta_n^{\star} - \theta)) dv,
    \end{equation}
    where
    \begin{equation}\label{eq:derivative.condition}
        T_n'(t) = \frac{1}{n} \sum_{i=1}^n \bb{1}_{\{X_i \neq t\}} h(X_i - t),
    \end{equation}
    and where $h : \R\backslash \{0\} \to \R$ is a measurable function (possibly nonlinear).
    In statistics, one is often interested in knowing if estimating a parameter ($\theta$ here) has an impact on the asymptotic law of a given statistic. See for example the interesting results of \cite{MR885745} in the context of limiting $\chi^2$ $U$ and $V$ statistics. Equations~\eqref{eq:Taylor.expansion} and \eqref{eq:derivative.condition} provide a natural setting for studying the question of whether or not $T_n(\theta_n^{\star}) - T_n(\theta) \to 0$ whenever $\theta_n^{\star}\rightarrow\theta$, as $n\to \infty$.

    Given some regularity conditions on the behavior of $h(\cdot)$ around the origin and in its tails, proving the convergence to $\EE[h(X_1 - \theta)]$, in probability say, of the integral on the right-hand side of \eqref{eq:Taylor.expansion} is often possible under weak assumptions by adapting standard uniform laws of large numbers. For instance, one can use \cite[Theorem~16~$(a)$]{MR1699953}, which was introduced by \cite{MR0054913} and \cite{MR0076246}.  One can also use entropy conditions: see, e.g., \cite[Chapter 3]{MR1739079} and \cite[Section 2.4]{MR1385671}.
    Some of these theorems go back to or evolved from the works of \cite{MR0070871}, \cite{MR0297000}, \cite{MR0288823,MR627861}, \cite{MR757767}, \cite{MR762984} and \cite{MR893902}.
    For extensive notes on the origins of the entropy conditions, we refer the interested reader to \cite[Section 3.8]{MR1739079} and \cite[pp.~36--38]{MR762984}.

    However, when $|h|$ blows up at $0$, namely when $\limsup_{x\to 0} |h(x)| = \infty$, these results are not applicable because the enveloppe function $h^{\text{sup}}(x) \circeq \sup_{t:|t - \theta| < \delta} \bb{1}_{\{x\neq t\}} |h(x - t)|$ is infinite in any small enough neighborhood of $\theta$ and, in particular, $h^{\text{sup}}(X_1)$ is not integrable for the outer measure.

    We faced such a problem when analysing the convergence of score functions in the context of testing the goodness-of-fit of the Laplace distribution with unknown location and scale parameters $(\mu,\sigma)$. If the family of alternatives is taken to be the asymmetric power distribution \citep{MR2395888} or the skewness exponential power distribution \citep{MR1379475}, a score function evaluated at the maximum likelihood estimator $(\mu_n^{\star},\sigma_n^{\star})$ can be used, in the spirit of \citep{MR3004652,Desgagne2017}.
    If the score function is expanded around $(\mu,\sigma)$, then a multivariate version of \eqref{eq:Taylor.expansion} is obtained.
    One of the integrals in the expansion will have an integrand \eqref{eq:derivative.condition} where $h(\cdot)$ contains a logarithmic term.
    Standard uniform laws of large numbers cannot be applied to show the convergence of such integrals because the enveloppe function of the class of functions $\{\log(\, \cdot \, - t)\}_{t:|t-\mu|<\delta}$ is infinite in any small enough neighborhood of $\mu$.
    In section \ref{sec:example}, we show how the main result of this paper (Theorem \ref{thm:probleme.Pierre.uniform}) can be used to prove a crucial part of the problem described above.

    More generally, the main result is that, under broad conditions, one obtains
    \begin{equation}\label{eq:thm:probleme.Pierre.uniform.conclusion.debut}
        \lim_{n\to\infty} \sup_{v\in [0,1]} \EE\left|\frac{1}{n} \sum_{i=1}^n \bb{1}_{\{X_i \neq \theta + v(\theta_n^{\star} - \theta)\}} h(X_i - \theta - v(\theta_n^{\star} - \theta)) - \EE\big[h(X_1 - \theta)\big]\right| = 0.
    \end{equation}
    From \eqref{eq:thm:probleme.Pierre.uniform.conclusion.debut} and the setting above, one can conclude that $T_n(\theta_n^{\star}) - T_n(\theta) \to 0$ in probability as $n\to \infty$.

\section{A new uniform \texorpdfstring{$L^1$}{L1} law of large numbers}\label{sec:new.ULLN}

        Throughout the paper, the labels $(X.k)$, $(H.k)$ and $(E.k)$ denote, respectively, assumptions that we will make on $X_1$, $h(\cdot)$ and $\theta_n$.
        Figure \ref{fig:proof.structure} at the end of the current section illustrates the logical structure of these assumptions and their implications.
        We start by proving a non-uniform version of Theorem \ref{thm:probleme.Pierre.uniform}.

        \begin{proposition}\label{prop:probleme.Pierre}
            Let $\theta\in \R$ and let $X_1, X_2, X_3, \ldots $ be a sequence of i.i.d.\hspace{-1mm} random variables such that
            \begin{description}
                \item[(X.1)] $\PP(X_1 = \theta) = 0$.
            \end{description}
            Let $h : \R\backslash\{0\} \to \R$ be a mesurable function that satisfies
            \begin{description}
                \item[(H.1)] $\PP(X_1 - \theta\in \mathcal{D}_h) = 0$, where $\mathcal{D}_h$ is the set of discontinuity points of $h(\cdot)$,
                \item[(H.2)] $\EE\, |h(X_1 - \theta)| < \infty$.
            \end{description}
            Let $\theta_n \circeq \theta_n(X_1,X_2,\ldots,X_n)$ be an estimator that satisfies
            \begin{description}
                \item[(E.1)] $\theta_n \stackrel{\PP}{\longrightarrow} \theta$,
                \item[(E.2)] \vspace{-1mm}For all $n\in \N$ and all $i\in \{1,2,\ldots,n\}$, $(X_i - \theta_n,X_i - \theta) \stackrel{\text{law}}{=} (X_1 - \theta_n,X_1 - \theta)$,
                \item[(E.3)] There exists $N_0\in \N$ such that $\big\{\bb{1}_{\{X_1 \neq \theta_n\}} h(X_1 - \theta_n)\big\}_{n\geq N_0}$ is uniformly integrable.
            \end{description}
            Then,
            \vspace{-1mm}
            \begin{equation}\label{eq:lem:probleme.Pierre.conclusion}
                \EE\left|\frac{1}{n} \sum_{i=1}^n \bb{1}_{\{X_i \neq \theta_n\}} h(X_i - \theta_n) - \EE\big[h(X_1 - \theta)\big]\right| \longrightarrow 0.
            \end{equation}
        \end{proposition}

        \begin{remark}
            Condition $(E.2)$ is satisfied for any estimator that is symmetric with respect to its $n$ variables. For example, this is the case for any maximum likelihood estimator that is based on i.i.d.\hspace{-1mm} observations.
        \end{remark}

        \begin{proof}[Proof of Proposition \ref{prop:probleme.Pierre}]
            From $(X.1)$ and $(E.1)$, we know that $\bb{1}_{\{X_1 = \theta_n\}} \hspace{-1mm}\stackrel{\PP}{\longrightarrow} 0$.
            Indeed, for any $\varepsilon > 0$,
            \begin{itemize}
                \item take $\delta \circeq \delta_{\varepsilon} > 0$ such that $\PP(|X_1 - \theta| < \delta) < \varepsilon/2$, and
                \item take $N \circeq N_{\delta,\varepsilon}$ such that for all $n\geq N$, we have $\PP(|\theta_n - \theta| \geq \delta) < \varepsilon/2$.
            \end{itemize}
            We get, for all $n\geq N$,
            \begin{equation*}
                \PP(X_1 = \theta_n) \leq \PP(X_1 = \theta_n, |\theta_n - \theta| < \delta) + \PP(|\theta_n - \theta| \geq \delta) < \varepsilon.
            \end{equation*}
            In particular, this shows $\bb{1}_{\{X_1 = \theta_n\}} |h(X_1 - \theta)| \stackrel{\PP}{\longrightarrow} 0$.
            Since this sequence is uniformly integrable by $(H.2)$, we also have the $L^1$ convergence.
            By using Jensen's inequality and $(E.2)$, we deduce
            \begin{equation}\label{eq:start.1.L1}
                \EE\, \left|\frac{1}{n} \sum_{i=1}^n \bb{1}_{\{X_i = \theta_n\}} h(X_i - \theta)\right| \leq \EE\big[\bb{1}_{\{X_1 = \theta_n\}} |h(X_1 - \theta)|\big] \longrightarrow 0.
            \end{equation}
            By $(H.2)$ and the law of large numbers in $L^1$ (see, e.g., Theorem 1.2.6 in \cite{MR2760872}), we also know that
            \begin{equation}\label{eq:start.2.L1}
                \EE\left|\frac{1}{n} \sum_{i=1}^n h(X_i - \theta) - \EE\big[h(X_1 - \theta)\big]\right| \longrightarrow 0.
            \end{equation}
            \newpage
            \noindent
            By combining \eqref{eq:start.1.L1} and \eqref{eq:start.2.L1}, we have shown
            \begin{equation}\label{eq:start.1.L1.2.L1.combined}
                \EE\left|\frac{1}{n} \sum_{i=1}^n \bb{1}_{\{X_i \neq \theta_n\}} h(X_i - \theta) - \EE\big[h(X_1 - \theta)\big]\right| \longrightarrow 0.
            \end{equation}
            To conclude the proof, we show that
            \begin{equation*}
                Y_n \circeq \frac{1}{n} \sum_{i=1}^n \bb{1}_{\{X_i \neq \theta_n\}} h(X_i - \theta_n) - \frac{1}{n} \sum_{i=1}^n \bb{1}_{\{X_i \neq \theta_n\}} h(X_i - \theta)  \stackrel{L^1}{\longrightarrow} 0.
            \end{equation*}
            From Jensen's inequality and $(E.2)$, we have
            \begin{equation}\label{eq:lem:probleme.Pierre.final.bound}
                \EE|Y_n| \leq \EE\Big[\bb{1}_{\{X_1 \neq \theta_n\}} \big|h(X_1 - \theta_n) - h(X_1 - \theta)\big|\Big].
            \end{equation}
            The sequence $\{\bb{1}_{\{X_1 \neq \theta_n\}} |h(X_1 - \theta_n) - h(X_1 - \theta)|\}_{n\in \N}$ converges to $0$ in probability by $(H.1)$, $(E.1)$ and the continuous mapping theorem \cite[Theorem 2.3]{MR1652247}.
            Furthermore, the sequence is uniformly integrable for $n \geq N_0$ by $(H.2)$, $(E.3)$ and the fact that the sums of random variables coming (respectively) from two uniformly integrable sequences form a uniformly integrable sequence.
            Hence, $Y_n \rightarrow 0$ in $L^1$.
        \end{proof}

        Since the distribution of $X_1 - \theta_n$ is rarely known, condition $(E.3)$ in Proposition \ref{prop:probleme.Pierre} is impractical to verify.
        The next lemma fix this problem.

        \begin{lemma}\label{lem:uniform.integrability}
            Let $\theta\in \R$. Let $X_1,X_2,X_3,\ldots$ be a sequence of i.i.d.\hspace{-1mm} random variables.
            Let $h : \R\backslash\{0\} \to \R$ be a mesurable function. Let $\theta_n \circeq \theta_n(X_1,X_2,\ldots,X_n)$ be an estimator that satisfies
            \begin{description}
                \item[(E.4)] If $\limsup_{x\to 0} |h(x)| < \infty$, we impose no condition. Otherwise, assume that there exist $N_1\in \N$, $\alpha_0 > 0$ and a constant $C_{\alpha_0} > 0$ such that
                    \vspace{-2mm}
                    \begin{equation*}
                        \sup_{n\geq N_1} \sup_{A\in \mathcal{B}_{>0}([-\alpha_0,\alpha_0])} \frac{\PP(X_1 - \theta_n\in A)}{\text{Lebesgue}(A)} \leq C_{\alpha_0} < \infty,
                    \end{equation*}
                    where $\mathcal{B}_{>0}([-\alpha_0,\alpha_0])$ denotes the Borel sets of positive Lebesgue measure on the interval $[-\alpha_0,\alpha_0]$.
                \item[(E.5)] There exist $N_2\geq 2$, $C, \gamma, p > 0$ and $\beta_0 > \gamma$ such that, for $\PP(X_1 - \theta\in \, \cdot\, )$-almost-all $x\in \R$, we have
                    \begin{itemize}
                        \item For all $u \geq (x + \gamma) \vee \beta_0$ and for all $n\geq N_2$,
                            \begin{align*}
                                \PP(\theta_n - \theta \leq x - u \nvert X_1 - \theta = x) \leq C e^{-|x-u|^{p}}.
                            \end{align*}
                        \item For all $u \leq (x - \gamma) \wedge (-\beta_0)$ and for all $n\geq N_2$,
                            \begin{align*}
                                \PP(\theta_n - \theta \geq x - u \nvert X_1 - \theta = x) \leq C e^{-|x-u|^{p}}.
                            \end{align*}
                    \end{itemize}
                \item[(E.6)] There exists $N_3\in \N$ such that for all $n \geq N_3$, there exists $A_n\in \mathcal{B}(\R)$ such that $\PP(X_1 - \theta\in A_n) = 1$ and, for all $x\in A_n$, the conditional measure $\PP(x - (\theta_n - \theta) \in \, \cdot\, \nvert X_1 - \theta = x)$, when restricted to $\{u\in \R : |u| \geq \beta_0, |x - u| > \gamma\}$, is absolutely continuous with respect to the Lebesgue measure.
            \end{description}
            Assume that $h(\cdot)$ satisfies
            \begin{description}
                \item[(H.3)] For all $x_0\in \R\backslash\{0\}$, $\limsup_{x\to x_0} |h(x)| < \infty$,
                \item[(H.4)] $\int_{|u| \leq \alpha_0} |h(u)| du < \infty$,
                \item[(H.5)]
                    \begin{description}
                        \item[1.] $h(\cdot)$ is absolutely continuous on bounded sub-intervals of $(-\infty,-\beta_0) \cup (\beta_0,+\infty)$;
                            \addtolength{\itemindent}{0.25cm}
                        \item[2.] There exists an integrable random variable $M$ such that $\sup_{|t| \leq \gamma}|h(X_1 - \theta - t)|\, \bb{1}_{\{|X_1 - \theta - t| \geq \beta_0\}} \leq M$ $\PP$-almost-surely;
                        \item[3.] $\lim_{|\beta|\to\infty} |h(\beta)| e^{-|x-\beta|^{p}} \hspace{-1mm}= 0$ for $\PP(X_1 - \theta\in \, \cdot\, )$-almost-all $x\in \R$, and $\{|h(\beta)| e^{-|X_1 - \theta - \beta|^{p}}\}_{|\beta| \geq \beta_0}$ is uniformly integrable;
                        \item[4.] $\int_{|u| \geq \beta_0} \EE\big[|h'(u)|\, e^{-|X_1 - \theta - u|^{p}}\big] du < \infty$;
                        \item[5.] For almost-all $|u| \geq \beta_0$, we have $-\text{sign}(u) \text{sign}(h(u)) h'(u) \leq 0$.
                    \end{description}
            \end{description}
            Then, $(E.3)$ from Proposition \ref{prop:probleme.Pierre} is satisfied, namely $\big\{\bb{1}_{\{X_1 \neq \theta_n\}} h(X_1 - \theta_n)\big\}_{n\geq N_0}$ is uniformly integrable, where $N_0 \circeq N_1 \vee N_2 \vee N_3$.
        \end{lemma}

        \begin{remark}\label{rem:lem:uniform.integrability}
            If $X_1 - \theta_n$ has a density for $n$ large enough and, in a neighborhood of $0$, those densities are uniformly bounded from above by the same positive constant, then $(E.4)$ is satisfied. In general, when $\theta_n$ is even only slightly non-trivial, we rarely know the distribution of $X_1 - \theta_n$.
            However, if $\theta_n$ concentrates more and more around $\theta$ as $n\to\infty$ (like most maximum likelihood estimators for instance), then we expect the weight of the distribution of $X_1$ around $\theta$ to dominate the weight of the distribution of $X_1 - \theta_n$ around $0$.
            In that case, we can expect $(E.4)$ to be satisfied when $X_1$ has a regular enough distribution around $\theta$.
            Condition $(E.5)$ is a way to control the tail behavior of $\theta_n$'s distribution for the above heuristic to work.
            Since the lemma is intended to be used when $|h|$ blows up at $0$, condition $(E.4)$ is there to control the distribution of $X_1 - \theta_n$ around $0$.
        \end{remark}

        \begin{proof}
            We want to prove that for $N_0 \circeq N_1 \vee N_2 \vee N_3$, we have
            \begin{equation*}
                \lim_{K\rightarrow\infty} \, \sup_{n\geq N_0} \EE\Big[\big|h(X_1 - \theta_n)\big|\, \bb{1}_{\{X_1 \neq \theta_n\} \cap \{|h(X_1 - \theta_n)| \geq K\}}\Big] = 0.
            \end{equation*}
            By $(H.3)$, $h(\cdot)$ is uniformly bounded on compact subsets of $\R\backslash \{0\}$.
            It is therefore sufficient to show both
            \begin{align}\label{eq:lem:uniform.integrability.to.show.1}
                &\lim_{\alpha\rightarrow 0} \, \sup_{n\geq N_0} \EE\Big[\big|h(X_1 - \theta_n)\big|\, \bb{1}_{\{X_1 \neq \theta_n\} \cap \{|X_1 - \theta_n| \leq \alpha\}}\Big] = 0, \\
                \label{eq:lem:uniform.integrability.to.show.2}
                &\lim_{\beta\rightarrow \infty} \, \sup_{n\geq N_0} \EE\Big[\big|h(X_1 - \theta_n)\big|\, \bb{1}_{\{|X_1 - \theta_n| \geq \beta\}}\Big] = 0.
            \end{align}
            When $\limsup_{x\to 0} |h(x)| < \infty$, then \eqref{eq:lem:uniform.integrability.to.show.1} is trivially satisfied because $h(\cdot)$ is uniformly bounded on compact subsets of $\R$ by $(H.3)$.
            When $\limsup_{x\to 0} |h(x)| = \infty$, then \eqref{eq:lem:uniform.integrability.to.show.1} follows directly from $(E.4)$, $(H.4)$ and the dominated convergence theorem (DCT).

            Assume for the remaining of the proof that
            \begin{equation*}
                n\geq N_0 \quad \text{and} \quad \beta > \beta_0 > \gamma,
            \end{equation*}
            where $\gamma$ and $\beta_0$ are fixed in $(E.5)$.
            Separate the expectation in \eqref{eq:lem:uniform.integrability.to.show.2} in two parts :
            \begin{align*}
                (a) + (b)
                &\circeq \EE\Big[\big|h(X_1 - \theta_n)\big|\, \bb{1}_{\{|X_1 - \theta_n| \geq \beta\} \cap \{|\theta_n - \theta| \leq \gamma\}}\Big] + \EE\Big[\big|h(X_1 - \theta_n)\big|\, \bb{1}_{\{|X_1 - \theta_n| \geq \beta\} \cap \{|\theta_n - \theta| > \gamma\}}\Big].
            \end{align*}
            By $(H.5).2$ and the DCT, we have $(a)\to 0$ as $\beta\to\infty$, uniformly in $n$.
            For the term $(b)$, condition on the value of $X_1 - \theta$, integrate by parts (see $(E.6)$ and $(H.5).1$) and then use $(E.5)$ and $(H.5).5$. We obtain
            \begin{align*}\label{eq:lem:uniform.integrability.1}
            (b)
            &= \int_{\{(u,x) \, : \, |u| \geq \beta, \, |x - u| > \gamma\}} |h(u)|\, \PP((X_1 - \theta_n,X_1 - \theta)\in d(u,x)) \notag \\[1mm]
            &= \int_{-\infty}^{\infty} \left(\int_{\{u \, : \, |u| \geq \beta, \, |x - u| > \gamma\}} |h(u)|\,
                \PP(x - (\theta_n - \theta) \in du \nvert X_1 - \theta = x)\right) \PP(X_1 - \theta\in dx) \notag \\[1mm]
            &= \int_{-\infty}^{-(\beta+\gamma)} \left\{\hspace{-1mm}
                \begin{array}{l}
                    \left.\Big[-|h(u)|\, \PP(\theta_n - \theta \leq x - u \nvert X_1 - \theta = x)\Big]\right|_{u=x+\gamma}^{-\beta} \\[3mm]
                    + \int_{x+\gamma}^{-\beta} \text{sign}(h(u))\, h'(u)\, \PP(\theta_n - \theta \leq x - u \nvert X_1 - \theta = x) \, du
                \end{array}
                \hspace{-1mm}\right\} \PP(X_1 - \theta\in dx) \notag \\[3mm]
            &+ \int_{-\infty}^{\infty} \, \lim_{t\to\infty} \left\{\hspace{-1mm}
                \begin{array}{l}
                    \left.\Big[-|h(u)|\, \PP(\theta_n - \theta \leq x - u \nvert X_1 - \theta = x)\Big]\right|_{u=(x+\gamma) \vee \beta}^t \\[3mm]
                    + \int_{(x+\gamma) \vee \beta}^t \text{sign}(h(u))\, h'(u)\, \PP(\theta_n - \theta \leq x - u \nvert X_1 - \theta = x) \, du \\[2mm]
                    + \left.\Big[|h(u)|\, \PP(\theta_n - \theta \geq x - u \nvert X_1 - \theta = x)\Big]\right|_{u=-t}^{(x-\gamma) \wedge (-\beta)} \\[3mm]
                    - \int_{-t}^{(x-\gamma) \wedge (-\beta)} \text{sign}(h(u))\, h'(u)\, \PP(\theta_n - \theta \geq x - u \nvert X_1 - \theta = x) \, du
                \end{array}
            \hspace{-1mm}\right\} \PP(X_1 - \theta\in dx) \notag \\[3mm]
            &+ \int_{\beta+\gamma}^{\infty} \left\{\hspace{-1mm}
                \begin{array}{l}
                    \left.\Big[|h(u)|\, \PP(\theta_n - \theta \geq x - u \nvert X_1 - \theta = x)\Big]\right|_{u=\beta}^{x - \gamma} \\[3mm]
                    - \int_{\beta}^{x - \gamma} \text{sign}(h(u))\, h'(u)\, \PP(\theta_n - \theta \geq x - u \nvert X_1 - \theta = x) \, du
                \end{array}
                \hspace{-1mm}\right\} \PP(X_1 - \theta\in dx) \notag \\[20mm]
            &\leq \int_{-\infty}^{-(\beta+\gamma)} \left\{\hspace{-1mm}
                \begin{array}{l}
                    |h(x+\gamma)| + 0
                \end{array}
                \hspace{-1mm}\right\} \PP(X_1 - \theta\in dx) \notag \\[1mm]
            &+ C \int_{-\infty}^{\infty} \left\{\hspace{-1mm}
                \begin{array}{l}
                    |h((x+\gamma) \vee \beta)|\, e^{-|x-((x+\gamma) \vee \beta)|^{p}} + \int_{\beta}^{\infty} |h'(u)|\, e^{-|x-u|^{p}} du \\[1mm]
                    |h((x-\gamma) \wedge (-\beta))|\, e^{-|x - ((x-\gamma) \wedge (-\beta))|^{p}} + \int_{-\infty}^{-\beta} |h'(u)|\, e^{-|x-u|^{p}} du
                \end{array}
                \hspace{-1mm}\right\} \PP(X_1 - \theta\in dx) \notag \\[1mm]
            &+ \int_{\beta+\gamma}^{\infty} \left\{\hspace{-1mm}
                \begin{array}{l}
                    |h(x-\gamma)| + 0
                \end{array}
                \hspace{-1mm}\right\} \PP(X_1 - \theta\in dx) \notag \\[1mm]
            &\lesssim \EE\Big[|h(X_1 - \theta + \gamma)| \bb{1}_{\{|X_1 - \theta + \gamma| \geq \beta\}}\Big] \,
                +\, \EE\Big[|h(\beta)|\, e^{-|X_1 - \theta - \beta|^{p}}\Big] ~+\, \int_{\beta}^{\infty} \EE\Big[|h'(u)|\, e^{-|X_1 - \theta - u|^{p}}\Big] du \notag \\
            &+ \EE\Big[|h(X_1 - \theta - \gamma)| \bb{1}_{\{|X_1 - \theta - \gamma| \geq \beta\}}\Big] \,
                +\, \EE\Big[|h(-\beta)|\, e^{-|X_1 - \theta + \beta|^{p}}\Big] ~+\, \int_{-\infty}^{-\beta} \EE\Big[|h'(u)|\, e^{-|X_1 - \theta - u|^{p}}\Big] du,
            \end{align*}
            where $y \lesssim z$ means $y \leq (1 \vee C) z$.
            As $\beta\rightarrow \infty$,
                    the first and fourth terms go to $0$ by $(H.5).2$ and the DCT,
                    the second and fifth terms go to $0$ by $(H.5).3$ and the DCT,
                    the third and sixth terms go to $0$ by $(H.5).4$ and the DCT.
            None of the terms depended on $n$, so the convergence is uniform in $n \geq N_0$.
        \end{proof}

        If $\{\theta_n^{\star}\}_{n\in \N}$ is a sequence of $M$-estimators, then the next lemma proposes an easy-to-verify condition on the tail probabilities of $\theta_n^{\star}$ for $(E.5)$ in Lemma \ref{lem:uniform.integrability} to hold uniformly in the set of estimators
        \begin{equation}\label{eq:T.class.estimators}
            \mathcal{E}_{n,\theta} \circeq \{\theta + v (\theta_n^{\star} - \theta)\}_{v\in [0,1]}, \quad \text{for some } \theta\in \R.
        \end{equation}

        \begin{lemma}\label{lem:m.estimator.exponential.bound}
            Let $\theta\in \R$ and let $X_1,X_2,X_3,\ldots$ be a sequence of i.i.d.\hspace{-1mm} random variables.
            Let $\{\theta_n^{\star}\}_{n\in \N}$ be a sequence of estimators satisfying
            \begin{equation}\label{eq:definition.M.estimator.1.to.n}
                \sum_{i=1}^n \psi(X_i - \theta_n^{\star}) = 0,
            \end{equation}
            where $\psi : \R \to \R$ is measurable, non-decreasing and $\psi(0) = 0$. Assume that there exist $N\geq 1$ and $C, \gamma, p \hspace{-0.3mm}>\hspace{-0.3mm} 0$ such that
            \begin{equation}\label{lem:m.estimator.exponential.bound.condition}
                \sup_{n\geq N} \PP\big(|\theta_n^{\star} - \theta| \geq |t|\big) \leq C e^{-|t|^{p}}, \quad \text{for all } |t| \geq \gamma.
            \end{equation}
            Then, condition $(E.5)$ from Lemma \ref{lem:uniform.integrability} is satisfied uniformly on $\mathcal{E}_{n,\theta}$, namely :
            \begin{description}
                \item[(E.5.\text{unif})] There exist $N_2\geq 2$, $C, \gamma, p > 0$ and $\beta_0 > \gamma$ such that, for $\PP(X_1 - \theta\in \, \cdot\, )$-almost-all $x\in \R$, we have
                    \begin{itemize}
                        \item For all $u \geq (x + \gamma) \vee \beta_0$ and for all $n\geq N_2$,
                            \begin{align*}
                                \sup_{\theta_n\in \mathcal{E}_{n,\theta}} \PP(\theta_n - \theta \leq x - u \nvert X_1 - \theta = x) \leq C e^{-|x-u|^{p}}.
                            \end{align*}
                        \item For all $u \leq (x - \gamma) \wedge (-\beta_0)$ and for all $n\geq N_2$,
                            \begin{align*}
                                \sup_{\theta_n\in \mathcal{E}_{n,\theta}} \PP(\theta_n - \theta \geq x - u \nvert X_1 - \theta = x) \leq C e^{-|x-u|^{p}}.
                            \end{align*}
                    \end{itemize}
            \end{description}
        \end{lemma}

        \begin{proof}
            For all $n \geq 2$, let $\theta^{\star}_{2:n} \circeq \theta^{\star}_{2:n}(X_2,X_3,\ldots,X_n)$ be an estimator that satisfies
            \begin{equation}\label{eq:definition.M.estimator.2.to.n}
                \sum_{i=2}^n \psi(X_i - \theta^{\star}_{2:n}) = 0 \quad \text{and} \quad \theta^{\star}_{2:n} \stackrel{\text{law}}{=} \theta^{\star}_{n-1}.
            \end{equation}
            Since $\psi$ is non-decreasing and $\psi(0) = 0$,
            \begin{align}
                \bullet \quad \theta_n^{\star} \leq X_1
                &~~\Longrightarrow~~ \psi(X_1 - \theta_n^{\star}) \geq 0
                ~~\stackrel{\eqref{eq:definition.M.estimator.1.to.n}}{\Longrightarrow}~~ \sum_{i=2}^n \psi(X_i - \theta_n^{\star}) \leq 0 ~~\stackrel{\eqref{eq:definition.M.estimator.2.to.n}}{\Longrightarrow}~~ \theta^{\star}_{2:n} \leq \theta_n^{\star} \leq X_1, \label{eq:lem:m.estimator.exponential.bound.end.eq.1.before} \\
                \bullet \quad \theta_n^{\star} \geq X_1
                &~~\Longrightarrow~~ \psi(X_1 - \theta_n^{\star}) \leq 0
                ~~\stackrel{\eqref{eq:definition.M.estimator.1.to.n}}{\Longrightarrow}~~ \sum_{i=2}^n \psi(X_i - \theta_n^{\star}) \geq 0 ~~\stackrel{\eqref{eq:definition.M.estimator.2.to.n}}{\Longrightarrow}~~ \theta^{\star}_{2:n} \geq \theta_n^{\star} \geq X_1. \label{eq:lem:m.estimator.exponential.bound.end.eq.2.before}
            \end{align}
            Let $\theta_n\in \mathcal{E}_{n,\theta}$ for all $n\in \N$.
            In order to prove \eqref{eq:lem:m.estimator.exponential.bound.end.eq.1} (respectively \eqref{eq:lem:m.estimator.exponential.bound.end.eq.2}) below, we use the following facts in succession : $\theta_n - \theta \leq 0 ~\Longrightarrow~ \theta_n^{\star} - \theta \leq \theta_n - \theta$ (respectively $\theta_n - \theta \geq 0 ~\Longrightarrow~ \theta_n^{\star} - \theta \geq \theta_n - \theta$), \eqref{eq:lem:m.estimator.exponential.bound.end.eq.1.before} (respectively \eqref{eq:lem:m.estimator.exponential.bound.end.eq.2.before}), the independence between $X_1$ and $\theta^{\star}_{2:n}$, \eqref{eq:definition.M.estimator.2.to.n}, and \eqref{lem:m.estimator.exponential.bound.condition}.

            \begin{itemize}
                \item For all $u \geq (x + \gamma) \vee \beta_0 > 0$ (note that $x - u \leq - \gamma < 0$) and for all $n\geq N + 1$, we have
                    \begin{align}\label{eq:lem:m.estimator.exponential.bound.end.eq.1}
                        \PP(\theta_n - \theta \leq x - u \nvert X_1 - \theta = x)
                        &\leq \PP(\theta_n^{\star} - \theta \leq x - u \nvert X_1 - \theta = x) \notag \\[1mm]
                        &\leq \PP(\theta_{2:n}^{\star} - \theta \leq x - u) \notag \\[1mm]
                        &= \PP(\theta_{n-1}^{\star} - \theta \leq x - u) \notag \\
                        &\leq C e^{-|x-u|^{p}} \hspace{-1mm}.
                    \end{align}
                \item For all $u \leq (x - \gamma) \wedge (-\beta_0) < 0$ (note that $x - u \geq \gamma > 0$) and for all $n\geq N + 1$, we have
                    \begin{align}\label{eq:lem:m.estimator.exponential.bound.end.eq.2}
                        \PP(\theta_n - \theta \geq x - u \nvert X_1 - \theta = x)
                        &\leq \PP(\theta_n^{\star} - \theta \geq x - u \nvert X_1 - \theta = x) \notag \\[1mm]
                        &\leq \PP(\theta_{2:n}^{\star} - \theta \geq x - u) \notag \\[1mm]
                        &= \PP(\theta_{n-1}^{\star} - \theta \geq x - u) \notag \\
                        &\leq C e^{-|x-u|^{p}} \hspace{-1mm}.
                    \end{align}
            \end{itemize}
            Simply choose $N_2 \circeq N + 1$ in $(E.5.\text{unif})$. This ends the proof.
        \end{proof}

        We can now state the main result. The structure of the assumptions is illustrated in Figure \ref{fig:proof.structure}.

        \begin{theorem}\label{thm:probleme.Pierre.uniform}
            Let $\theta\in \R$ and let $X_1,X_2,X_3,\ldots$ be a sequence of i.i.d.\hspace{-1mm} random variables satisfying
            \begin{description}
                \item[(X.1)] $\PP(X_1 = \theta) = 0$.
            \end{description}
            Let $\{\theta_n^{\star}\}_{n\in \N}$ be a sequence of estimators satisfying {\bf (E.5.unif)} directly or the conditions in Lemma \ref{lem:m.estimator.exponential.bound}.
            Denote $\mathcal{E}_{n,\theta} \circeq \{\theta + v (\theta_n^{\star} - \theta)\}_{v\in [0,1]}$, and assume that
            \begin{description}
                \item[(E.1.\text{unif})] $\theta_n^{\star} \stackrel{\PP}{\longrightarrow} \theta$;
                \item[(E.2.\text{unif})]\vspace{-0.8mm}
                    For all $n\in \N$, all $i\in \{1,2,\ldots,n\}$ and all $\theta_n\in \mathcal{E}_{n,\theta}$, $(X_i - \theta_n,X_i - \theta) \stackrel{\text{law}}{=} (X_1 - \theta_n,X_1 - \theta)$;
                \item[(E.4.\text{unif})] If $\limsup_{x\to 0} |h(x)| < \infty$, we impose no condition.
                    Otherwise, assume that there exist $N_1\in \N$, $\alpha_0 > 0$ and a constant $C_{\alpha_0} > 0$ such that
                    \begin{equation*}
                        \sup_{n\geq N_1} \sup_{\theta_n\in \mathcal{E}_{n,\theta}} \sup_{A\in \mathcal{B}_{>0}([-\alpha_0,\alpha_0])} \frac{\PP(X_1 - \theta_n\in A)}{\text{Lebesgue}(A)} \leq C_{\alpha_0} < \infty.
                    \end{equation*}
                \item[(E.6.\text{unif})] There exists $N_3\in \N$ such that for all $n \geq N_3$ and for all $\theta_n\in \mathcal{E}_{n,\theta}$, there exists $A_{n,\theta_n}\in \mathcal{B}(\R)$ such that $\PP(X_1 - \theta\in A_{n,\theta_n}) = 1$ and, for all $x\in A_{n,\theta_n}$, the measure $\PP(x - (\theta_n - \theta) \in \, \cdot\, \nvert X_1 - \theta = x)$, when restricted to $\{u\in \R : |u| \geq \beta_0, |x - u| > \gamma\}$, is absolutely continuous with respect to the Lebesgue measure.
            \end{description}
            Finally, assume
            \begin{description}
                \item[(H.1), (H.2)] from Proposition \ref{prop:probleme.Pierre},
                \item[(H.3), (H.4), (H.5)] from Lemma \ref{lem:uniform.integrability}.
            \end{description}
            Then, the conclusion in Proposition \ref{prop:probleme.Pierre} holds uniformly for $\theta_n\in \mathcal{E}_{n,\theta}$, namely
            \begin{equation}\label{eq:thm:probleme.Pierre.uniform.conclusion}
                \lim_{n\to\infty} \sup_{\theta_n\in \mathcal{E}_{n,\theta}} \EE\left|\frac{1}{n} \sum_{i=1}^n \bb{1}_{\{X_i \neq \theta_n\}} h(X_i - \theta_n) - \EE\big[h(X_1 - \theta)\big]\right| = 0.
            \end{equation}
        \end{theorem}

        \begin{proof}
            We know that $(E.5.\text{unif})$ holds, either directly or via the conditions in Lemma \ref{lem:m.estimator.exponential.bound}.
            By combining $(E.4.\text{unif})$ to $(E.6.\text{unif})$ and $(H.3)$ to $(H.5)$, a proof along the lines of Lemma \ref{lem:uniform.integrability} shows
            \begin{description}
                \item[(E.3.\text{unif})]
                \begin{equation*}\label{eq:uniform.integrability.uniform.in.S}
                    \lim_{K\rightarrow\infty} \, \sup_{n\geq N_0} \sup_{\theta_n \in \mathcal{E}_{n,\theta}} \EE\bigg[\big|h(X_1 - \theta_n)\big|\, \bb{1}_{\{X_1 \neq \theta_n\} \cap \{|h(X_1 - \theta_n)| \geq K\}}\bigg] = 0.
                \end{equation*}
            \end{description}
            By $(E.3.\text{unif})$, $(H.2)$ and the identity $|U_n + V_n| \bb{1}_{\{|U_n + V_n| \geq 2K\}} \leq 2 |U_n| \bb{1}_{\{|U_n| \geq K\}} + 2 |V_n| \bb{1}_{\{|V_n| \geq K\}}$, we deduce
            \begin{equation}\label{eq:uniform.integrability.uniform.in.S.consequence}
                \lim_{K\rightarrow\infty} \, \sup_{n\geq N_0} \sup_{\theta_n \in \mathcal{E}_{n,\theta}}
                \EE\bigg[\hspace{-1mm}
                    \begin{array}{l}
                        |h(X_1 - \theta_n) - h(X_1 - \theta)| \bb{1}_{\{X_1 \neq \theta_n\} \cap \{|h(X_1 - \theta_n) - h(X_1 - \theta)| \geq K\}}
                    \end{array}
                    \hspace{-1mm}\bigg] = 0.
            \end{equation}

            To conclude, we rerun the proof of Proposition \ref{prop:probleme.Pierre} with our new assumptions.
            By $(X.1)$, $(H.2)$, $(E.1.\text{unif})$ and $(E.2.\text{unif})$, the convergence in \eqref{eq:start.1.L1} is valid for $\sup_{\theta_n \in \mathcal{E}_{n,\theta}}$ of the expectation.
            This implies that the convergence in \eqref{eq:start.1.L1.2.L1.combined} is also valid for $\sup_{\theta_n \in \mathcal{E}_{n,\theta}}$ of the expectation.
            Furthermore, by $(H.1)$, $(E.1.\text{unif})$ and the continuous mapping theorem, we have, for all $\varepsilon > 0$,
            \begin{equation}\label{eq:thm:probleme.Pierre.uniform.end.uniform.convergence.probability}
                \lim_{n\to\infty} \sup_{\theta_n\in \mathcal{E}_{n,\theta}} \PP\Big(\bb{1}_{\{X_1 \neq \theta_n\}} \big|h(X_1 - \theta_n) - h(X_1 - \theta)\big| > \varepsilon\Big) = 0.
            \end{equation}
            By combining \eqref{eq:uniform.integrability.uniform.in.S.consequence} and \eqref{eq:thm:probleme.Pierre.uniform.end.uniform.convergence.probability}, the $\sup_{\theta_n \in \mathcal{E}_{n,\theta}}$ of the expectation on the right-hand side of \eqref{eq:lem:probleme.Pierre.final.bound} converges to $0$.
            In summary, we have shown that $\sup_{\theta_n \in \mathcal{E}_{n,\theta}}$ of the expectations in \eqref{eq:start.1.L1}, \eqref{eq:start.1.L1.2.L1.combined} and \eqref{eq:lem:probleme.Pierre.final.bound} all converge (respectively) to $0$. Hence, the conclusion of Proposition \ref{prop:probleme.Pierre} holds for $\sup_{\theta_n \in \mathcal{E}_{n,\theta}}$ of the expectation, which is exactly the claim made in \eqref{eq:thm:probleme.Pierre.uniform.conclusion}.
        \end{proof}

        \begin{remark}
            By following the proof of Theorem \ref{thm:probleme.Pierre.uniform}, we see that $(X.1)$, $(H.1)$, $(H.2)$, $(E.1.\text{unif})$, $(E.2.\text{unif})$ and $(E.3.\text{unif})$ alone imply the conclusion in \eqref{eq:thm:probleme.Pierre.uniform.conclusion}. The other assumptions in the statement of the theorem are simply there to give a more practical way to verify $(E.3.\text{unif})$.
        \end{remark}

        \vspace{8mm}
        \begin{figure}[ht]
            \centering
            \includegraphics[scale=1.15]{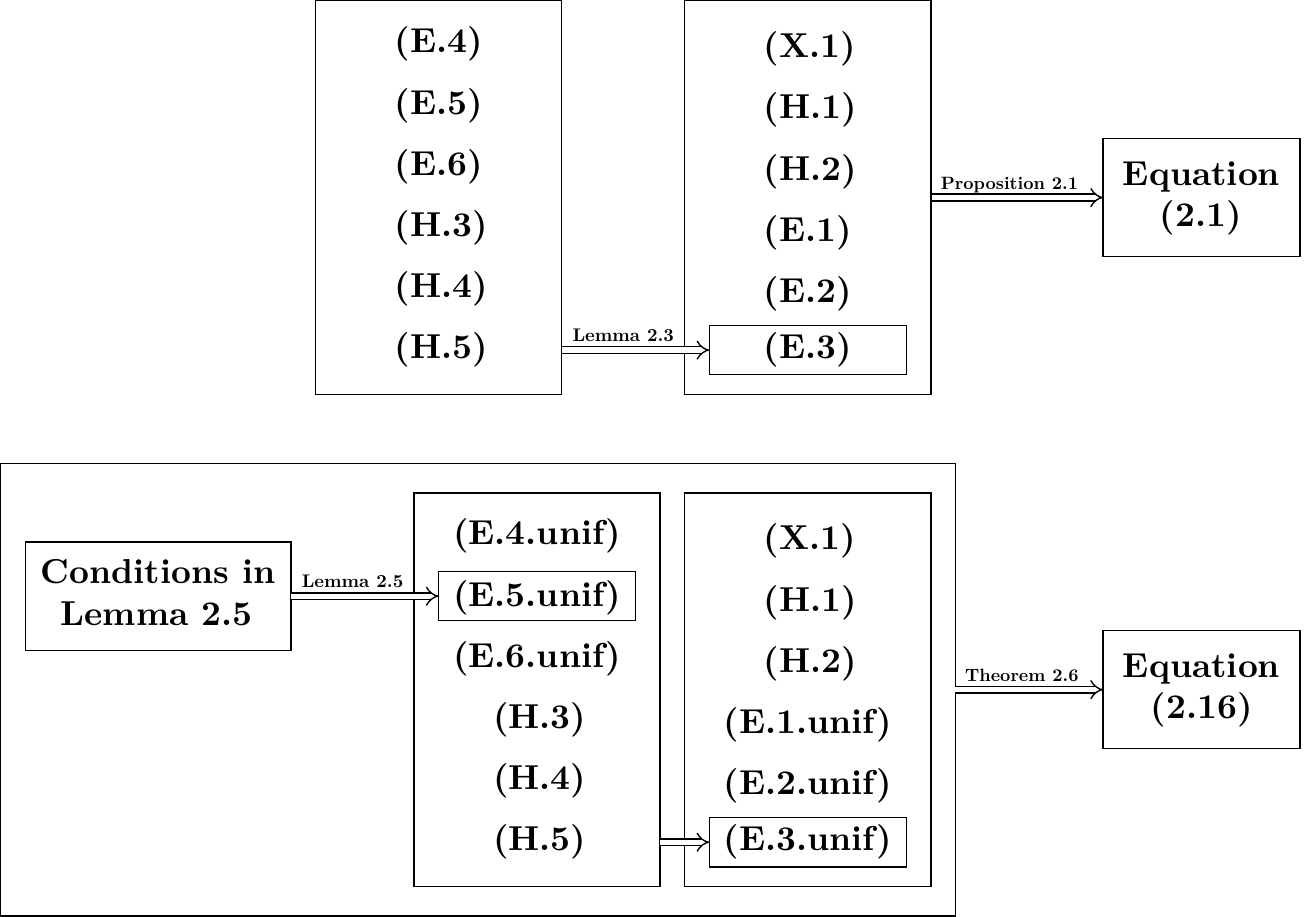}
            \caption{Logical structure of the assumptions and their implications.}
            \label{fig:proof.structure}
        \end{figure}

\section{Example}\label{sec:example}

    We now give an application of the previous theorem.
    The context of the problem is described at the end of Section \ref{sec:intro}.

        \newpage
        \begin{lemma}\label{lem:example.1}
            Let $X_1,X_2,X_3,\ldots$ be a sequence of i.i.d.\hspace{-0.5mm} random variables with density function
            \begin{equation*}
                f_{X_1}(x) \circeq \frac{1}{4 \sigma} e^{-\frac{1}{2} \left|\frac{x - \mu}{\sigma}\right|}, \quad x\in \R,
            \end{equation*}
            where $\mu\in \R$ and $\sigma > 0$.
            Define $h : \R\backslash\{0\} \to \R$ by
            \begin{equation*}
                h(y) \circeq \text{sign}(y) \log|y|.
            \end{equation*}
            Let
            \begin{equation}\label{eq.MLEs.laplace.example}
                \mu_n^{\star} \circeq \text{median}(X_1,X_2,\ldots,X_n)
                \circeq
                \left\{\hspace{-1mm}
                \begin{array}{ll}
                    X_{((n+1)/2)}, ~&\mbox{if } n ~\text{is odd}, \\[1mm]
                    \frac{1}{2} (X_{(n/2)} + X_{(n/2+1)}), ~&\mbox{if } n ~\text{is even}.
                \end{array}
                \right.
            \end{equation}
            For $v\in [0,1]$, define $\mu_{n,v}^{\star} \circeq \mu + v(\mu_n^{\star} - \mu)$, and let $\mathcal{E}_{n,\mu} \circeq \{\mu_{n,v}^{\star}\}_{v\in [0,1]}$.
            Then,
            \begin{equation}\label{eq:lem:example.1.conclusion}
                \lim_{n\to\infty} \sup_{v\in [0,1]} \EE\left|\frac{1}{n} \sum_{i=1}^n \bb{1}_{\{X_i \neq \mu_{n,v}^{\star}\}} h(X_i - \mu_{n,v}^{\star}) - \EE\left[h(X_1 - \mu)\right]\right| = 0.
            \end{equation}
        \end{lemma}

        \begin{proof}
            Without loss of generality, assume that $\mu = 0$.
            Below, we verify the conditions of Theorem \ref{thm:probleme.Pierre.uniform}.
            \begin{description}
                \item[(X.1)] $\PP(X_1 = 0) = 0$. This is obvious.
                \item[\bf(Conditions in Lemma \ref{lem:m.estimator.exponential.bound})]
                    We show that the conditions are satisfied with $\psi(y) \circeq \text{sign}(y)$ and $\psi(0) \circeq 0$.
                    Indeed, by \eqref{eq.MLEs.laplace.example}, we know that $\sum_{i=1}^n \psi(X_i - \mu_n^{\star}) = 0$.
                    Furthermore, for $N\in \N$ and $\gamma > 0$ both large enough (depending on $\sigma$), we have, for all $n \geq N$ and all $t \geq \gamma$,
                    \begin{align}\label{eq:lem:probleme.Pierre.uniform.first.condition}
                        \PP(\mu_n^{\star} \geq t)
                        &\leq \sum_{k=\lceil n/2 \rceil}^n \binom{n}{k} \, \PP(X_1 \geq t)^k \, \PP(X_1 \leq t)^{n - k} \notag \\
                        &\leq (n - \lceil n/2 \rceil) \cdot \binom{n}{\lceil n/2 \rceil} \cdot \PP(X_1 \geq t)^{\lceil n/2 \rceil} \notag \\
                        &\leq \lfloor n/2 \rfloor \cdot 2 \frac{2^n}{\sqrt{n}} \cdot \Big(\frac{1}{2} e^{-\frac{t}{2\sigma}}\Big)^{\lceil n/2 \rceil}
                        \leq \frac{\sqrt{n}}{2} 2^n e^{-\frac{nt}{8\sigma}} \cdot e^{-\frac{nt}{8\sigma}}
                        \leq \frac{1}{2} e^{-t}.
                    \end{align}
                    To obtain the third inequality, we use Stirling's formula and assume that $N$ is large enough.
                    To obtain the last inequality, assume that $N \geq 8\sigma$ and $\gamma \geq 8 \sigma$.
                    This proves \eqref{lem:m.estimator.exponential.bound.condition} with $C = 1$ and $p = 1$.
                \item[(E.1.\text{unif})] $\mu_n^{\star} \stackrel{\PP}{\longrightarrow} 0$. This is explained in Example 5.11 of \cite{MR1652247}.
                \item[(E.2.\text{unif})] For any $v\in [0,1]$, the estimator $\mu_{n,v}^{\star} = v \mu_n^{\star}$ is symmetric with respect to its $n$ variables because the median, $\mu_n^{\star}$, is symmetric with respect to its $n$ variables. Since the $X_i$'s are i.i.d., the condition is satisfied.
                \item[(E.4.\text{unif})] We have $\limsup_{x\to 0} |h(x)| = \infty$, so we need to verify the condition.
                    For any $n\geq 2$ and any $v\in [0,1]$, note that $X_1 - v \mu_n^{\star}$ has a density function.
                    It suffices to show that the densities are bounded, uniformly in $n$ and $v$, by a positive constant.
                    Since the density $u \mapsto f_{X_1 - v \mu_n^{\star}}(u)$ is symmetric around $0$, we will assume, without loss of generality, that $u >0$.
                    For $v\in (0,1]$, denote $z \circeq (x - u) / v$ and notice that $z < x$.

                    When $v\in (0,1]$ and $n\geq 3$ is odd, we have
                    \begin{align*}
                        f_{X_1 - v \mu_n^{\star}}(u)
                        &= \int_{-\infty}^{\infty} f_{X_1 - v\mu_n^{\star} | X_1}(u \nvert x) f_{X_1}(x) dx
                        = \int_{-\infty}^{\infty} \frac{1}{v} f_{\mu_n^{\star} | X_1}(z \nvert x) f_{X_1}(x) dx \\
                        &= \int_{-\infty}^{\infty}  \frac{1}{v} \binom{n}{\lfloor n/2 \rfloor} (F_{X_1}(z))^{\lfloor n/2 \rfloor} f_{X_1}(z) (1 - F_{X_1}(z))^{\lfloor n/2 - 1 \rfloor} f_{X_1}(x) dx \\
                        &\leq C\, \|f_{X_1}\|_{\infty}  \underbrace{\int_{-\infty}^{\infty} \frac{1}{v} f_{\mu_{n-2}^{\star}}(z) dx}_{=~1} \\
                        &= C\, \|f_{X_1}\|_{\infty} < \infty.
                    \end{align*}
                    In the inequality above, we took $C \circeq \sup_{n\geq 3} \binom{n}{\lfloor n/2 \rfloor} / \binom{n-2}{\lfloor (n-2)/2 \rfloor}$, which is finite by Stirling's formula.
                    When $v\in (0,1]$ and $n\geq 4$ is even, we can apply a similar argument and also obtain a uniform bound.
                    Finally, when $v = 0$ and $n\in \N$, $f_{X_1 - v \mu_n^{\star}}(u) = f_{X_1}(u) \leq 1/(4\sigma)$.

                    In summary, $f_{X_1 - v \mu_n^{\star}}(u)$ is uniformly bounded in $u\in \R$, $n \geq 3$ and $v\in [0,1]$, which proves $(E.4.\text{unif})$ with any $\alpha_0 > 0$ and any $N_1 \geq 3$.
                \item[(E.6.\text{unif})] In our case, this is trivial because the conditional density $f_{X_1 - v \mu_n^{\star}|X_1}(\cdot \nvert x)$ exists for all $x\in \R$, all $n \geq 2$ and all $v\in (0,1]$.
                \item[(H.1)] The function $h$ is continuous on $\R\backslash \{0\}$, so $\mathcal{D}_h = \emptyset$ and thus $\PP(X_1\in \mathcal{D}_h) = 0$.
                \item[(H.2)] $\EE\big|h(X_1)\big| \leq \int_{|x| \leq 1} |\log|x|| \frac{1}{4\sigma} dx + \int_{|x| \geq 1} |x| f_{X_1}(x) dx \leq \frac{2}{4\sigma} + 2\sigma < \infty$.
                \item[(H.3)] For all $x_0\in \R\backslash\{0\}$, $\limsup_{x\to x_0} |h(x)| < \infty$. This is obvious.
                \item[(H.4)] $\int_{|u| \leq \alpha_0} |\log|u|| du < \infty$ is true for any $\alpha_0 > 0$ since $\int_{|u| \leq 1} |\log|u|| du = 2$.
                \item[(H.5)]
                    \begin{enumerate}
                        \item This is obviously true for any $\beta_0 > 0$ (use the fundamental theorem of calculus).
                        \item For any $\gamma > 0$ and any $\beta_0 > \gamma$, the supremum $\sup_{|t| \leq \gamma} |h(X_1 - t)|\bb{1}_{\{|X_1 - t| \geq \beta_0\}}$ is attained at the boundary with probability $1$ (not necessarily the same end of the boundary for different $\omega$'s). Therefore, take $M = |h(X_1 - \gamma)|\bb{1}_{\{|X_1 - \gamma| \geq \beta_0\}} + |h(X_1 + \gamma)|\bb{1}_{\{|X_1 + \gamma| \geq \beta_0\}}$. It is easy to show that $\EE[M] < \infty$ because $|\log|x|| \leq |x|$ for $|x| \geq 1$ and $\int_{|x| \geq (1 \vee \beta_0)} |x| f_{X_1 \pm \gamma}(x) dx < \infty$.
                        \item We need to verify this condition for $p = 1$ since this is the $p$ that we used above to verify the conditions of Lemma \ref{lem:m.estimator.exponential.bound}.
                            First, $\lim_{|\beta|\to \infty} |h(\beta)| e^{-|x - \beta|^p} = 0$ is true for all $x\in \R$ and all $p > 0$ (true in particular for $p = 1$).
                            For the second part, assume that $\beta \geq 1$. We have
                            \begin{align}\label{eq:prop:example.1.H.6}
                                \EE[e^{-|X_1 - \beta|}]
                                &= \int_{(-\infty,0) \cup (0,\beta) \cup (\beta,\infty)} e^{-|x-\beta|} \cdot \frac{1}{4\sigma} e^{-\frac{1}{2\sigma} |x|} dx \notag \\
                                &\leq \frac{1}{2} e^{-|\beta|} \underbrace{\int_{-\infty}^0 \frac{1}{2\sigma} e^{-\frac{1}{2\sigma} |x|} dx}_{=~1} \, + \, \frac{|\beta|}{4\sigma} e^{-(1 \wedge \frac{1}{2\sigma}) |\beta|} + \frac{1}{4\sigma} e^{-\frac{1}{2\sigma} |\beta|} \underbrace{\int_{\beta}^{\infty} e^{-|x - \beta|} dx}_{=~1} \notag \\
                                &\leq \frac{|\beta|}{2} \left(1 \vee \frac{1}{2\sigma}\right) e^{-(1 \wedge \frac{1}{2\sigma})|\beta|}.
                            \end{align}
                            By the symmetry of $f_{X_1}$, we also have \eqref{eq:prop:example.1.H.6} for $\beta \leq -1$.
                            Hence, for any $\beta_0 \geq 1$,
                            \begin{equation*}
                                \sup_{|\beta| \geq \beta_0} \EE\left[\Big(|h(\beta)| e^{-|X_1 - \beta|}\Big)^2\right] < \infty,
                            \end{equation*}
                            which is a well-known sufficient condition for the uniform integrability of $\{|h(\beta)| e^{-|X_1 - \beta|}\}_{|\beta| \geq \beta_0}$, see e.g.\hspace{-0.5mm} \cite[Corollary 6.21]{MR3112259}.
                        \item Take any $\beta_0 \geq 1$, then \eqref{eq:prop:example.1.H.6} implies
                            \begin{equation*}
                                \int_{|u| \geq \beta_0} \EE\big[|h'(u)| e^{-|X_1 - u|}\big] du \leq \frac{1}{\beta_0} \int_{|u| \geq \beta_0} \EE\big[e^{-|X_1 - u|}\big] du < \infty.
                            \end{equation*}
                        \item Take any $\beta_0 \geq 1$, then, for all $|u| \geq \beta_0$,
                            \begin{equation*}
                                -\text{sign}(u) \cdot \text{sign}(h(u)) \cdot h'(u)
                                = -\text{sign}(u) \cdot \text{sign}(u) \cdot \frac{1}{|u|} \leq 0.
                            \end{equation*}
                    \end{enumerate}
            \end{description}
            This ends the proof.
        \end{proof}

\section*{Acknowledgements}

We would like to thank the anonymous reviewer for helpful comments that contributed to improving the final version of the paper.


\bibliographystyle{elsarticle-harv}

\begin{thebibliography}{0}
\expandafter\ifx\csname natexlab\endcsname\relax\def\natexlab#1{#1}\fi
\expandafter\ifx\csname url\endcsname\relax
  \def\url#1{\texttt{#1}}\fi
\expandafter\ifx\csname urlprefix\endcsname\relax\def\urlprefix{URL }\fi

\end{thebibliography}


\begin{thebibliography}{20}
\expandafter\ifx\csname natexlab\endcsname\relax\def\natexlab#1{#1}\fi
\expandafter\ifx\csname url\endcsname\relax
  \def\url#1{\texttt{#1}}\fi
\expandafter\ifx\csname urlprefix\endcsname\relax\def\urlprefix{URL }\fi

\bibitem[{Blum(1955)}]{MR0070871}
Blum, J.~R., 1955. On the convergence of empiric distribution functions. Ann.
  Math. Statist. 26, 527--529,
  \href{http://www.ams.org/mathscinet-getitem?mr=MR0070871}{MR0070871}.

\bibitem[{de~Wet and Randles(1987)}]{MR885745}
de~Wet, T., Randles, R.~H., 1987. On the effect of substituting parameter
  estimators in limiting {$\chi^2\;U$} and {$V$} statistics. Ann. Statist.
  15~(1), 398--412,
  \href{http://www.ams.org/mathscinet-getitem?mr=MR885745}{MR885745}.

\bibitem[{Dehardt(1971)}]{MR0297000}
Dehardt, J., 1971. Generalizations of the {G}livenko-{C}antelli theorem. Ann.
  Math. Statist. 42, 2050--2055,
  \href{http://www.ams.org/mathscinet-getitem?mr=MR0297000}{MR0297000}.

\bibitem[{Desgagn\'e and {Lafaye de Micheaux}(2018)}]{Desgagne2017}
Desgagn\'e, A., {Lafaye de Micheaux}, P., 2018. A powerful and interpretable
  alternative to the {Jarque-Bera} test of normality based on 2nd-power
  skewness and kurtosis, using the {R}ao's score test on the {APD} family. J.
  Appl. Stat.~(To appear)
  \href{https://doi.org/10.1080/02664763.2017.1415311}{doi:10.1080/02664763.2017.1415311}.

\bibitem[{Desgagn\'e et~al.(2013)Desgagn\'e, Lafaye~de Micheaux, and
  Leblanc}]{MR3004652}
Desgagn\'e, A., Lafaye~de Micheaux, P., Leblanc, A., 2013. Test of normality
  against generalized exponential power alternatives. Comm. Statist. Theory
  Methods 42~(1), 164--190,
  \href{http://www.ams.org/mathscinet-getitem?mr=MR3004652}{MR3004652}.

\bibitem[{Ferguson(1996)}]{MR1699953}
Ferguson, T.~S., 1996. A {C}ourse in {L}arge {S}ample {T}heory. Texts in
  Statistical Science Series. Chapman \& Hall, London,
  \href{http://www.ams.org/mathscinet-getitem?mr=MR1699953}{MR1699953}.

\bibitem[{Fern\'andez et~al.(1995)Fern\'andez, Osiewalski, and
  Steel}]{MR1379475}
Fern\'andez, C., Osiewalski, J., Steel, M. F.~J., 1995. Modeling and inference
  with {$v$}-spherical distributions. J. Amer. Statist. Assoc. 90~(432),
  1331--1340,
  \href{http://www.ams.org/mathscinet-getitem?mr=MR1379475}{MR1379475}.

\bibitem[{Gin\'e and Zinn(1984)}]{MR757767}
Gin\'e, E., Zinn, J., 1984. Some limit theorems for empirical processes. Ann.
  Probab. 12~(4), 929--998,
  \href{http://www.ams.org/mathscinet-getitem?mr=MR757767}{MR757767}.

\bibitem[{Klenke(2014)}]{MR3112259}
Klenke, A., 2014. Probability theory, 2nd Edition. Universitext. Springer,
  London, \href{http://www.ams.org/mathscinet-getitem?mr=MR3112259}{MR3112259}.

\bibitem[{Komunjer(2007)}]{MR2395888}
Komunjer, I., 2007. Asymmetric power distribution: theory and applications to
  risk measurement. J. Appl. Econometrics 22~(5), 891--921,
  \href{http://www.ams.org/mathscinet-getitem?mr=MR2395888}{MR2395888}.

\bibitem[{LeCam(1953)}]{MR0054913}
LeCam, L., 1953. On some asymptotic properties of maximum likelihood estimates
  and related {B}ayes' estimates. Univ. California Publ. Statist. 1, 277--329,
  \href{http://www.ams.org/mathscinet-getitem?mr=MR0054913}{MR0054913}.

\bibitem[{Pollard(1984)}]{MR762984}
Pollard, D., 1984. Convergence of {S}tochastic {P}rocesses. Springer Series in
  Statistics. Springer-Verlag, New York,
  \href{http://www.ams.org/mathscinet-getitem?mr=MR762984}{MR762984}.

\bibitem[{Rubin(1956)}]{MR0076246}
Rubin, H., 1956. Uniform convergence of random functions with applications to
  statistics. Ann. Math. Statist. 27, 200--203,
  \href{http://www.ams.org/mathscinet-getitem?mr=MR0076246}{MR0076246}.

\bibitem[{Stroock(2011)}]{MR2760872}
Stroock, D.~W., 2011. Probability {T}heory, 2nd Edition. Cambridge University
  Press, Cambridge,
  \href{http://www.ams.org/mathscinet-getitem?mr=MR2760872}{MR2760872}.

\bibitem[{Talagrand(1987)}]{MR893902}
Talagrand, M., 1987. The {G}livenko-{C}antelli problem. Ann. Probab. 15~(3),
  837--870, \href{http://www.ams.org/mathscinet-getitem?mr=MR893902}{MR893902}.

\bibitem[{van~de Geer(2000)}]{MR1739079}
van~de Geer, S.~A., 2000. Applications of {E}mpirical {P}rocess {T}heory.
  Vol.~6 of Cambridge Series in Statistical and Probabilistic Mathematics.
  Cambridge University Press, Cambridge,
  \href{http://www.ams.org/mathscinet-getitem?mr=MR1739079}{MR1739079}.

\bibitem[{van~der Vaart(1998)}]{MR1652247}
van~der Vaart, A.~W., 1998. Asymptotic {S}tatistics. Vol.~3 of Cambridge Series
  in Statistical and Probabilistic Mathematics. Cambridge University Press,
  Cambridge,
  \href{http://www.ams.org/mathscinet-getitem?mr=MR1652247}{MR1652247}.

\bibitem[{van~der Vaart and Wellner(1996)}]{MR1385671}
van~der Vaart, A.~W., Wellner, J.~A., 1996. Weak {C}onvergence and {E}mpirical
  {P}rocesses. Springer Series in Statistics. Springer-Verlag, New York,
  \href{http://www.ams.org/mathscinet-getitem?mr=MR1385671}{MR1385671}.

\bibitem[{Vapnik and {\v C}ervonenkis(1971)}]{MR0288823}
Vapnik, V.~N., {\v C}ervonenkis, A.~Y., 1971. The uniform convergence of
  frequencies of the appearance of events to their probabilities. Teor.
  Verojatnost. i Primenen. 16, 264--279,
  \href{http://www.ams.org/mathscinet-getitem?mr=MR0288823}{MR0288823}.

\bibitem[{Vapnik and {\v C}ervonenkis(1981)}]{MR627861}
Vapnik, V.~N., {\v C}ervonenkis, A.~Y., 1981. Necessary and sufficient
  conditions for the uniform convergence of empirical means to their true
  values. Teor. Veroyatnost. i Primenen. 26~(3), 543--563,
  \href{http://www.ams.org/mathscinet-getitem?mr=MR627861}{MR627861}.

\end{thebibliography}

\end{document}